\documentclass{iopconfser}
\usepackage{amsmath, amssymb, amsthm, amsfonts, graphicx}

\theoremstyle{definition}
\newtheorem{definition}{Definition}[section]
\newtheorem{theorem}{Theorem}[section]
\newtheorem{corollary}{Corollary}[theorem]

\newtheorem{proposition}{Proposition}[section]

\newtheorem{remark}{Remark}[section]

\begin{document}

\title{Way-below relation and tensor products}

\author{Cristian Ivanescu$^{1}$ and Hunter Labrecque$^{2}$}

\affil{$^1$Department of Mathematics, MacEwan University, Edmonton, Canada}
\affil{$^2$ Department of Mathematics, MacEwan University, Edmonton, Canada} 

\email{ivanescuc@macewan.ca}

\begin{abstract}
We establish that the way-below relation is preserved under the tensor product. After completing this work, we became aware that, in the framework of the Cuntz semigroup, this result has already been observed in the literature. Nevertheless, we include our argument here, as we believe it offers a complementary perspective and may assist the reader in better understanding the behaviour of the way-below relation in this setting.

\end{abstract}

\section{Introduction} 

 The way-below relation, initially defined for open set inclusion, see \cite{Scott}, captures a more nuanced relationship than simple containment.  The statement \textit{open set U is contained in open set V} is often too broad.  For example, $V$ containing $U$ plus a single additional point is fundamentally different from $U$ being a subset of $V$ where the closure of $U$ is entirely within $V$.  In locally compact spaces, the relevant setting for our work, we refine this notion.  We require that the closure of $U$ be compact and contained within $V$. This condition, known as \textit{compact containment} or \textit{the way-below relation}, provides crucial additional information. For visualizations of compact containment of sets, please refer to Appendix B, Figure 2. Figure 1 provides one-dimensional and two-dimensional examples of subsets that are not compactly contained. Importantly, this topological concept has an analogue in the category of C*-algebras, making it a valuable tool in both contexts.\\
 
This paper is organized as follows: Section 2 introduces the noncommutative definition of the way-below relation. In the context of the Cuntz semigroup, we utilize countably generated Hilbert C$^*$-modules. This assumption is crucial because it allows us to apply the Kasparov stabilization theorem. Therefore, the more precise terminology should technically be \textit{countably generated compact containment}. This definition still makes sense in the setting of any Hilbert C*-modules, not just countable generated Hilbert C*-modules. We will henceforth assume all Hilbert C*-modules are countably generated as we are interested in the setting of the Cuntz semigroup.\\


 Section 3 lays the groundwork for the main result by examining the way-below relation for ideals within commutative C$^*$-algebras. Leveraging the correspondence between ideals and open sets in this context, we demonstrate the preservation of the way-below relation under tensor products. 
Section 4 presents the paper's main results. We offer two distinct proofs: the first relies on the definition of way-below relation using Hilbert C*-modules as outlined in Theorem 4.1, while the second employs a characterization of the way-below relation provided by Gardella and Perera (Theorem 4.2).

\section{Non-commutative definition of compact containment}

In \cite{CEI}, a specific relation is established within the context of countably generated Hilbert C$^*$-modules. On page 168 of this paper, it is demonstrated that this concrete relation is equivalent to the abstract, order-theoretic concept of compact containment of the corresponding Cuntz equivalence classes. We now recall the definition of compact containment (also known as the way-below relation) for Hilbert C$^*$-modules, which serves as a non-commutative analogue of the way-below relation for open sets in topology. 
See \cite{Lan} or \cite{Man} for details about Hilbert C$^*$-modules.

\begin{definition}
 Let $X$ and $Y$ be two countably generated Hilbert C$^*$-modules. We say that $X$ is countably compact contained in $Y$,  $ X \subset \subset Y$, if there is a compact self-adjoint endomorphism $b$ of $Y$ which acts as the identity on $X$.  
\end{definition}
We also recall the purely order-theoretic definition, which is applicable to any ordered set.
\begin{definition}
 We say that $x << y$ in a given ordered set if whenever $(y_n)$ is such that $\sup y_n \geq y$ then there is some $i_0$ such that $x \leq y_{i_0}$.   
\end{definition}

Since closed two sided ideals can naturally be viewed as Hilbert C$^*$-modules, we make the following definition:
\begin{definition} Let $I,J$ be two closed two sided ideals of a C$^*$-algebra $A$. We say that ideal $I$ is compactly contained in ideal $J$, written $I \subset \subset J$, if the Hilbert C$^*$-module $I$ is compactly contained in the Hilbert C$^*$-module $J$; see Definition 2.1 above.     
\end{definition}

\begin{remark}
In \cite{Gab}, J. Gabe's Definition 10 introduces a concept of compact containment for ideals in complete lattices. We hypothesize that Gabe's definition is equivalent to our proposed definition of compact containment; however, a formal proof of this equivalence remains elusive. We aim to explore this relationship further in subsequent research. 
\end{remark}

It is shown in \cite{Scott}, Proposition I-1.4, together with the Remark right after Definition I-1.1, that in the context of open subsets of a topological space $K$, the order-theoretic definition of the way-below relation $ U << V$ is equivalent to the existence of a compact set $Q$ satisfying $U \subseteq Q \subseteq V$. Moreover, when $K$ is Hausdorff, this is equivalent to the closure of $U$, $\overline{U}$, being compact and contained within $V$.\\

We define the way-below relation for open sets for the reader's reference.
For open sets $U,V$ of a locally compact set $K$: 
\begin{definition}
 $U << V$ if $\overline{U} \subseteq V$ and  $\overline{U}$ is compact.
\end{definition}

\section{Way-below relation and tensor products: the commutative case}

This section explores the concept of compact containment in commutative C$^*$-algebras and demonstrates its preservation under tensor products.\\

The structure theorem of Gelfand and Naimark establishes a fundamental duality between commutative C*-algebras and topological spaces. Specifically, a commutative unital C$^*$-algebra is isomorphic to the algebra of continuous functions, $C(K)$, on a compact Hausdorff space $K$. Similarly, a non-unital commutative C$^*$-algebra is isomorphic to $C_0(L)$, the algebra of continuous functions vanishing at infinity on a locally compact Hausdorff space $L$. The process of minimal unitization for non-unital algebras $C_0(L)$ corresponds to the one-point compactification of the locally compact space $L$.\\

Next we extend the established duality to encompass compact containment, proving its validity for compact containment of ideals in $C(K)$ and their associated open subsets.\\

It is worth noting that this duality motivates the interpretation of non-commutative C*-algebras as analogous to $C(N)$, where $N$ represents a non-commutative topological space. However, it's crucial to acknowledge that within the study of non-commutative C*-algebras, the concept of such a \textit{non-commutative space} $N$ lacks a universally accepted, concrete definition.\\

\begin{remark}
There exists a correspondence between open subsets $U \subseteq K$ and closed two-sided ideals in $C(K)$, where $K$ is compact. Specifically, open sets $U$ define ideals of functions vanishing outside $U$, and conversely, every closed two-sided ideal $I$ is expressible as $C_0(U)$ for some open $U \subseteq K$. (See Appendix A, Proposition, for a complete account.)
\end{remark}
\begin{proposition}
Let $I_1$ and $I_2$ be closed ideals in $C(K)$, compactly contained in $J_1$ and $J_2$, respectively. Then $I_1 \otimes I_2$ is compactly contained in $J_1 \otimes J_2$. 
\end{proposition}
\begin{proof}

This can be approached via compact containment of open sets. First, we present a useful remark:\\
\begin{remark}
  Given ideals $I,J \subset C(K)$ with $I \subset \subset J$, we represent them as $I=C_0(U)$ and $J=C_0(V)$ for open $U,V \subseteq K$. Applying Definition 2.1, we find that the identity on $C_0(U)$ is a compact endomorphism. Consequently, the closure of $U$ must be a compact subset of $V$, leading to the conclusion that $U$ is compactly contained in $V$.  
\end{remark}

 Hence the compact containment of closed ideals is equivalent to the compact containment of their corresponding open sets.  That is, $I_1 \subset \subset J_1 $
  if and only if $U_1 << V_1$
 , where $U_1, V_1$ denote the corresponding open sets of ideals $I_1, J_1$. Note that $$I_1 \otimes I_2 \cong C_0(U_1) \otimes C_0(U_2) \cong C_0(U_1 \times U_2)$$
 and similarly 
 $$J_1 \otimes J_2 \cong C_0(V_1) \otimes C_0(V_2) \cong C_0(V_1 \times V_2)$$
 Therefore, showing $I_1 \otimes I_2 \subset \subset J_1 \otimes J_2$ 
 is equivalent to showing that $U_1 \times U_2 << V_1 \times V_2$.\\ 

 A straightforward exercise in topology demonstrates the following: If $U_1$ is compactly contained in $V_1$, and $U_2$ is compactly contained in $V_2$, then their Cartesian product, $U_1 \times U_2$, is compactly contained in $V_1 \times V_2$. This follows directly from the definition of compact containment.  Indeed, our assumptions imply that the closures of $U_1$ and $U_2$ are compact sets contained within $V_1$ and $V_2$, respectively.  Since the Cartesian product of compact sets is itself compact, and because the closure of $U_1 \times U_2$ is equal to the Cartesian product of the closures of $U_1$ and $U_2$, i.e.,  
 $$\overline{U_1 \times U_2}=\overline{U_1} \times \overline{U_2}$$ we conclude that $U_1 \times U_2$ is compactly contained in $V_1 \times V_2$ 
\end{proof}
 
This result extends readily to Cartesian products of finitely many pairs of compactly contained sets. Furthermore, by leveraging Tychonoff's theorem, we can generalize it to countably many such pairs.  In the context of ideals, this infinite product scenario corresponds to the inductive limit of ideals.

 \section{Main Result}
In this section, we present our main result. Throughout this section we assume that all C$^*$-algebras have stable rank one, all Hilbert C$^*$-modules are countably generated and all hereditary C$^*$-subalgebras are $\sigma$-unital. Assuming the C$^*$-algebra has stable rank one makes it possible (see \cite{CEI}) to view the Cuntz relation as isomorphisms between associated Hilbert C$^*$-modules. The assumption that hereditary C$^*$-subalgebras are $\sigma$-unital is a standard assumption in many contexts. This is equivalent to the algebra containing a strictly positive element. A key class of C$^*$-algebras with this property includes separable C$^*$-algebras. For a detailed understanding of compact self-adjoint endomorphisms, we refer the reader to \cite{Man} or \cite{Lan}. Notably, the proof we provide closely parallels that of Theorem 1 in \cite{CEI}, with the key distinction being the replacement of direct sums with tensor products. We are grateful to George Elliott for bringing this observation to our attention.\\

Convention on Tensor Products: If the C$^*$-algebras in question are not nuclear, the choice of tensor product becomes significant. However, the arguments presented here are robust; they hold irrespective of the specific tensor product chosen, whether it be the spatial tensor product, the maximal tensor product, or any other intermediate tensor product.\\

Unless stated otherwise, the term \textit{tensor product} in this paper refers to the minimal, or spatial, tensor product. For nuclear C$^*$-algebras, the spatial tensor product is unique and coincides with all other possible C$^*$-tensor products, including the maximal tensor product. All commutative C$^*$-algebras are nuclear.

\begin{theorem} 
The compact containment property is preserved under tensor products for countably generated Hilbert C$^*$-modules. Specifically, if 
$$Y_1 \subset \subset X_1\;\mathrm{and}\; Y_2 \subset \subset X_2\;\mathrm{then} $$
    $$Y_1 \otimes Y_2 \subset \subset X_1 \otimes X_2.$$
\end{theorem}

\begin{proof} We begin with some essential definitions and properties.
Given a Hilbert $A$-module $X$ and a Hilbert $B$-module $Y$, their external tensor product $X\otimes Y$ is defined as the completion of the algebraic tensor product $X \odot Y $ with respect to the norm induced by the $A\otimes B$-valued inner product
$$<x_1 \otimes y_1, x_2 \otimes y_2>=<x_1,x_2><y_1,y_2> $$ resulting in a Hilbert $A \otimes B$-module.\\

 Let $b_1$ and $b_2$ be compact self-adjoint endomorphisms on modules $X_1$ and $X_2$, respectively. Assume that $b_1$ acts as identity on  $Y_1 \subseteq X_1$, and $b_2$ acts as identity on $Y_2 \subseteq X_2$. Then the tensor product of $b_1 \otimes b_2$ is a compact self-adjoint endomorphism on the $X_1 \otimes X_2$, and acts as an identity on the tensor product $Y_1 \otimes Y_2$. 
\end{proof}

\begin{definition}
   Let $A$ be a C$^*$-algebra and $M$ be a Hilbert C$^*$-module over $A$. We say that a Hilbert C$^*$-module $M$ is a compact element if $M  \subset \subset  M$.
\end{definition}
 \begin{corollary}
 If $M_1$ and $M_2$ compact (i.e., $M_1 \subset \subset M_1$, $M_2  \subset \subset M_2$ then $M_1 \otimes M_2$ is compact (i.e., $M_1 \otimes M_2  \subset \subset M_1 \otimes M_2$.)   
 \end{corollary}

\begin{corollary}
 Assume $A$ is a C$^*$-algebra and $I_{1,2}$, $J_{1,2}$ are ideals such that $I_1 \subset \subset J_1$ and  $I_2 \subset \subset J_2$. Then $$I_1 \otimes I_2 \subset \subset J_1 \otimes J_2$$
\end{corollary}

\subsection{Compact containment in the Cuntz semigroup}
This result is based on a private discussion between the first author and George Elliott, whom both authors thank for the valuable insight.\\ 
After completing the proof of the following theorem (Theorem 4.2), we became aware that a similar result had been obtained independently by Antoine, Perera, and Thiel; see Paragraph 6.4.10 in \cite{APT}. For the reader’s convenience, we nevertheless include our proof, as it offers an alternative approach and may provide additional insight into the statement. \\
In what follows we assume $A$ is a separable nuclear C$^*$-algebra and $a_1, a_2, b_1 a_2 \in A$.\\
\indent We demonstrate that compact containment in the Cuntz semigroup is preserved under the tensor product. Specifically: 
\begin{theorem}
If the Cuntz class of $a_1$ is compactly contained in that of $b_1$, $[a_1] << [b_1]$, and the same holds for $a_2$ and $b_2$, $[a_2] << [b_2]$,  then the Cuntz class of their tensor product, $[a_1 \otimes a_2]$, is compactly contained in that of $[b_1 \otimes b_2]$:
$$[a_1 \otimes a_2] << [b_1 \otimes b_2].$$
\end{theorem}

\begin{proof}

The proof relies on two main ingredients.\\
\indent We use the fact from Proposition 4.3 \cite{GarPer} that $[a] << [b]$ if and only if there exists an $\epsilon >0$ such that $[a]$ is dominated by the $\epsilon$ cutdown of $b$: $$[a] << [b] \Longleftrightarrow [a] \leq [(b-\epsilon)_+].$$
 We refer to the above fact as the Epsilon-Cutdown Characterization.\\
 
 In particular, $[a]$ is compact if and only if there exists $\epsilon >0$ such that $a \sim (a-\epsilon)_+$.
A second important fact we rely on is the following commutative case result (see \cite{GarPer}, Proposition 4.4). For positive functions $a$ and $b$, the relation $[a] << [b]$  is equivalent to the condition that the closure of the open support of $a$ is contained in the open support of $b$:
$$ \overline{\mathrm{supp}_0(a)} \subseteq \mathrm{supp}_0 (b)$$
where $\mathrm{supp}f_0=\{x \in X: f(x) \neq 0\}$, and $f:X \rightarrow \mathbb{C}$, $X$ a compact Hausdorff space. Since $f$ is assumed to be a positive function (as in the case of $a$ and $b$ above), it must be self-adjoint, meaning $f$ takes real values, and positivity further ensures that these values are nonnegative real numbers.\\

We begin with the assumption that the Cuntz class $[a_1]$  is compactly contained in $[b_1]$  and $[a_2]$  is compactly contained in $[b_2]$. Our goal is to show that $[a_1 \otimes a_2]$ is compactly contained in $[b_1 \otimes b_2]$.\\
 By the Epsilon-Cutdown characterization of compact containment (previously mentioned as the first fact, from Proposition 4.3, \cite {GarPer}), our assumption implies the existence of $\epsilon_1 >0$ and $\epsilon_2 >0$ such that the Cuntz classes are dominated by their respective cutdowns:
    $$[a_1] \leq [(b_1 -\epsilon_1)_+] \;\mathrm{and}\; [a_2] \leq [(b_2-   \epsilon_2)_+]$$
    This implies that the Cuntz class of the tensor product is also dominated:
    $$[a_1 \otimes a_2] \leq [(b_1 - \epsilon_1)_+ \otimes (b_2 -\epsilon_2)_+]$$
We use the property that if $[a_1] \leq [b_1]$ and $[a_2] \leq [b_2]$, where $a_1, a_2, b_1, b_2$ are all positive, then $[a_1 \otimes a_2]  \leq [b_1 \otimes b_2]$. We include the straightforward proof below to make the paper easier to read.\\
Given that $a_1=\lim v_n b_1 v_n^*$  ( since, $[a_1]$ is Cuntz less than $[b_1]$) and $a_2=\lim w_n b_2 w_n^*$, (since $[a_2]$ is Cuntz less than $[b_2],$ we can deduce the following using the property $(x\otimes y)(a \otimes b)=xa \otimes yb$:

$a_1 \otimes a_2=\lim (v_n b_1 v_n^*) \otimes (w_n b_2 w_n^*) =\lim (v_n \otimes w_n)(b_1 \otimes b_2)(v_n^* \otimes w_n^*)$. This shows that $[a_1 \otimes a_2] \leq [b_1 \otimes b_2].$

Now, we establish a key inequality. The element $(b_1-\epsilon_1)_+$  belongs to the commutative C$^*$-algebra generated by $b_1$, which is $C(\sigma (b_1))$. Similarly, $(b_2-\epsilon_2)_+$  is in $C(\sigma(b_2))$. The tensor product of these commutative algebras is also commutative: 
$$C(\sigma(b_1))\otimes C(\sigma(b_2)) \cong C(\sigma(b_1) \times \sigma(b_2)).$$ 

We make the following substitutions: $f_1=(b_1- \epsilon_1)_+$, $f_2=(b_2- \epsilon_2)_+$, $g_1=b_1$ and $g_2=b_2$. We then aim to show that $f_1f_2$ is way below $g_1g_2$. This proof uses a commutative version of the Gardella and Perera condition which is: $f <<g$  if there is $c>0$ such that $f(x)>0$ then $g(x)>c$, combined with the result of the following proposition.

\begin{proposition}
Let $K_1$ and $K_2$ be compact Hausdorff spaces, and let 
\[
f_1, g_1 \in C(K_1)_+, \qquad f_2, g_2 \in C(K_2)_+.
\]
Assume that $f_1 \ll g_1$ and $f_2 \ll g_2$ in the sense of Gardella--Perera.  
Then the function
\[
h(x,y) = f_1(x) f_2(y)
\]
is way-below
\[
k(x,y) = g_1(x) g_2(y)
\]
in $C(K_1 \times K_2)_+$.
\end{proposition}

\begin{proof}
Recall that in the commutative case, the Gardella--Perera characterization of the way-below relation for positive continuous functions states that
\[
f \ll g \quad \Longleftrightarrow \quad 
\exists\, c>0 \text{ such that } f(x)>0 \Rightarrow g(x)>c.
\]
Applying this to each pair $(f_i, g_i)$, we obtain constants $c_1, c_2 > 0$ such that
\[
f_1(x)>0 \ \Rightarrow\ g_1(x)>c_1, 
\qquad 
f_2(y)>0 \ \Rightarrow\ g_2(y)>c_2.
\]

Let $c = c_1 c_2 > 0$.  
Suppose $h(x,y) = f_1(x) f_2(y) > 0$.  
Then both $f_1(x)>0$ and $f_2(y)>0$, and hence $g_1(x)>c_1$ and $g_2(y)>c_2$.  
Therefore,
\[
k(x,y) = g_1(x) g_2(y) > c_1 c_2 = c.
\]
This shows that
\[
h(x,y)>0 \ \Rightarrow\ k(x,y)>c.
\]
By the commutative version of Gardella--Perera criterion again, this is precisely the condition $h \ll k$.
\end{proof}

    Let's choose a single $\epsilon =\epsilon_1\epsilon_2$.
    Within the commutative algebra $C(\sigma (b_1)\times \sigma (b_2))$, the open support of the function $(b_1 \otimes b_2-\epsilon)_+$  contains the product of the individual open supports of $(b_1-\epsilon_1)_+$  and $(b_2-\epsilon_2)_+$.

    According to the characterization for commutative elements (the second fact mentioned, from [\cite{GarPer}, Proposition 4.4]), this support containment implies the following domination:
    $$[(b_1-\epsilon_1)_+ \otimes (b_2-\epsilon_2)_+]\leq [(b_1 \otimes b_2-\epsilon)_+] $$

    By combining the results above, we create a chain of domination:
    $$[a_1 \otimes a_2] \leq [(b_1-\epsilon_1)_+\otimes (b_2-\epsilon_2)_+]\leq [(b_1 \otimes b_2-\epsilon)_+]$$
    The existence of an $\epsilon>0$ that satisfies $[a_1 \otimes a_2] \leq [(b_1 \otimes b_2-\epsilon)_+]$  is precisely the condition for compact containment.

    Therefore, we conclude that $[a_1 \otimes a_2]$  is compactly contained in $[b_1 \otimes b_2]$.

\end{proof}
\section{Appendix A} While the content of this section is likely familiar to experts in the field, we believe it offers a convenient reference for all readers, contributing to a more comprehensive understanding of the paper's overall context.\\
\indent This section establishes a fundamental link between closed ideals of the algebra of continuous functions $C(X)$ and the underlying compact Hausdorff space $X$. Specifically, it demonstrates that all closed (two-sided) ideals within $C(X)$ can be uniquely characterized by a corresponding closed subset of $X$ or equivalently by an open subset of $X$. It is important to note that within the context of a compact Hausdorff space, closed subsets inherently possess the compactness property.\\
\indent Let $X$ be a metrizable compact Hausdorff topological space. We denote by $C(X)$ the algebra of all continuous complex-valued functions on $X$. This algebra is separable, meaning it contains a countable dense subset. In the case of $C(X)$ with $X=[0,1]$, the set of all polynomials with rational coefficients forms such a countable dense subset.\\

{\bf Proposition} If $X$ is compact and $I$ is a closed ideal of $C(X)$, then there is a closed subset $F$ of $X$ such that 
$$I=\{ f \in C(X): f(x)=0, \forall x \in F\}$$

{\bf Proof} Let $I$ be a proper closed ideal in $C(X)$. Define $$ F = \{x \in X: f(x)=0, \forall f \in I\}$$
We will show that $F$ is both closed and nonempty.\\
\indent We observe that $$F=\bigcap_{f \in I} ker f$$ which using the fact that an arbitrary intersection of closed sets is closed implies that $F$ is a closed set. Note that since $f$ is continuous, $ker f$ is a closed set. \\
\indent Next, we argue that such $F$ is a nonempty set. We show that if $f,g \in I$ are arbitrary, then $ker f \cap ker g \neq \emptyset$. \\
\indent Assume to obtain a contradiction that $ker f \cap ker g = \emptyset$. Because $|f|^2=(ff^*)$ still belongs to the ideal $I$, we can assume without loss of generality that both $f$ and $g$ are positive. Then $f+g$ is strictly positive in $X$ and $f+g$ belongs to the ideal. This in turn implies that $I$ has an invertible element which shows that $I=C(X)$, i.e., $I$ is not a proper ideal, which is a contradiction. \\
\indent It has been deduced that a closed ideal $I$ in $C(X)$ is precisely the set of continuous functions that are zero on a closed subset $F$ of $X$. An alternative characterization of this ideal is the set of continuous functions whose support is contained within the open complement $U=X-F$:
           $$I=C_0(X \setminus F)=C_0(U).$$
\section{Appendix B}
\begin{center}
\includegraphics[scale=0.5]{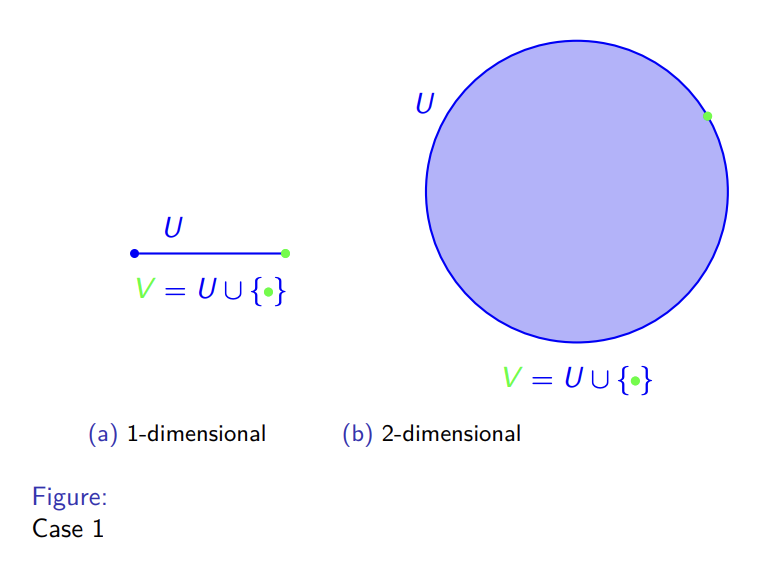} 
\includegraphics[scale=0.5]{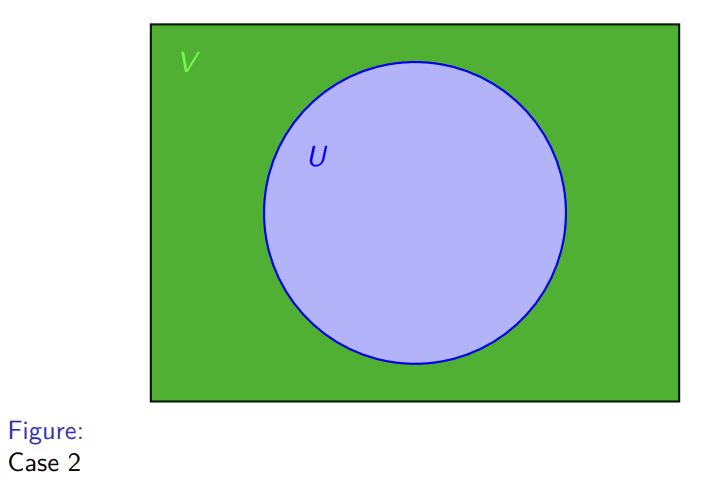}

\end{center}

\end{document}